\documentclass[a4paper, 10pt, preprint]{elsarticle}
\usepackage{amsmath, amssymb, amsthm, dsfont, a4}
\usepackage{hyperref}
\usepackage[all]{xy}
\usepackage[active]{srcltx}
\usepackage[short,nodayofweek]{datetime}

\newcommand\FF{{\mathbb F}}
\newcommand\GG{{\mathbb G}}
\newcommand\NN{{\mathbb N}}
\newcommand\CC{{\mathbb C}}
\newcommand\TT{{\mathbb T}}

\newcommand\pitilde{\tilde{\pi}}

\newcommand{\mot}{\mathsf{M}}  
\newcommand{\dumot}{\mathfrak{M}}  

\newcommand{\transp}[1]{#1^{\rm tr}}   

\newcommand{\eps}{\varepsilon}
\newcommand{\rank}{\mathrm{rk}_{\FF_q[t]}}
\newcommand{\bigmid}{\, \Big| \,}
\newcommand{\trk}{t\textrm{-}\mathrm{rk}}

\newcommand{\dphi}{d\phi}
\newcommand{\hd}[2]{\partial_t^{(#1)}\!\!\left(#2\right)}
\newcommand{\hde}[1]{\partial_t^{(#1)}}  
\newcommand{\e}{\exp_E}  
\newcommand{\id}{\mathrm{id}}
\newcommand{\res}{\mathrm{res}_{t=\theta}}
\newcommand{\ev}{\mathrm{ev}_1}

\newcommand{\tate}{\CC_\infty\cs{t}}  
\newcommand{\laurent}{\CC_\infty(\!(t)\!)}  
\newcommand{\powser}{\CC_\infty[\![t]\!]}  
\newcommand{\ps}[1]{[\![#1]\!]}  
\newcommand{\ls}[1]{(\!(#1)\!)}
\newcommand{\cs}[1]{\langle #1 \rangle} 

\newcommand{\vect}[1]{\text{\boldmath $#1$\unboldmath}} 
\newcommand{\svect}[2]{\left( \begin{matrix} {#1}_{1}\\ \vdots \\ {#1}_{#2}\end{matrix}\right)}

\newcommand{\betr}[1]{\lvert #1\rvert}
\newcommand{\norm}[1]{\lVert #1\rVert}

\DeclareMathOperator{\Hom}{Hom}
\DeclareMathOperator{\End}{End}

\DeclareMathOperator{\Mat}{Mat}
\DeclareMathOperator{\GL}{GL}
\DeclareMathOperator{\Lie}{Lie}

\DeclareMathOperator{\Ker}{ker}
\DeclareMathOperator{\Coker}{coker}

\theoremstyle{plain}
\newtheorem{thm}{Theorem}[section]
\newtheorem{cor}[thm]{Corollary}
\newtheorem{lem}[thm]{Lemma}
\newtheorem{prop}[thm]{Proposition}

\newtheorem*{mainthm}{Main Theorem}

\theoremstyle{definition}
\newtheorem{defn}[thm]{Definition}
\newtheorem{exmp}[thm]{Example}
\newtheorem{rem}[thm]{Remark}

\begin{document}

\title{Periods of $t$-modules as special values}
\author[am]{Andreas Maurischat}
\ead[am]{andreas.maurischat@matha.rwth-aachen.de}
\address[am]{Lehrstuhl A f\"ur Mathematik, RWTH Aachen University, Germany }

\date{\formatdate{08}{08}{2018}}



\begin{abstract}
In this article we show that all periods of uniformizable $t$-modules (resp.~their coordinates) can be obtained via specializing a rigid analytic trivialization of a related dual $t$-motive at $t=\theta$. The proof is even constructive.
The central object in the construction is a subset $H$ of the Tate algebra points of $E$ which turns out to be isomorphic to the period lattice of $E$ via kind of generating series in one direction and residues in the other. This isomorphism even holds for arbitrary $t$-modules $E$, even non-abelian ones.
\end{abstract}

\maketitle

\setcounter{tocdepth}{1}
\tableofcontents

\section{Introduction}\label{sec:intro}


Questions on algebraic independence of periods of $t$-modules are of great interest in number theory in positive characteristic. The most prominent period, the Carlitz period 
\[ \tilde{\pi}=\lambda_\theta \theta \prod_{j \geq 1} (1 - \theta^{1-q^j})^{-1} \in K_\infty(\lambda_\theta), \]
is the function field analog of the complex number $2\pi i$, and it was already proven 
by Wade in 1941 that $\tilde{\pi}$ is transcendental over $K$ (see \cite{liw:cqtg}).
Here, $K$ is a finite extension of the rational function field $\FF_q(\theta)$ over the finite field $\FF_q$,  $K_\infty$ is the completion of $K$ with respect to the absolute value
$|\cdot|_\infty$ given by $|\theta|_\infty=q$, and $\lambda_\theta\in \bar{K}$ is a $(q-1)$-th root of $-\theta$
in an algebraic closure $\bar{K}$ of $K$.

For proving algebraic independence of periods (and other ``numbers'' like zeta values and logarithms), the ABP-criterion (cf.~\cite[Thm.~3.1.1]{ga-wb-mp:darasgvpc}) and a consequence of it - which is part of the proof of \cite[Thm.~5.2.2]{mp:tdadmaicl} - as well as a generalization of Chang (cf.~\cite[Thm.~1.2]{cc:nrvabpc}) turned out to be very useful. To state the consequence given by Papanikolas, let $\CC_\infty$ denote the completion of the algebraic closure of $K_\infty$, and
$\powser$ the power series ring over $\CC_\infty$, as well as 
$\tate$ the subring consisting of those power series which converge on the closed unit disc $|t|_\infty\leq 1$. 
Finally, let $\mathbb{E}$ be the subring of entire functions, i.e.~of those power series which converge for all $t\in \CC_\infty$ and whose coefficients lie in a finite extension of $K_\infty$.
On $\tate$ we consider the inverse Frobenius twist $\sigma$ given by
\[ \sigma( \sum_{i=0}^\infty x_it^i)=\sum_{i=0}^\infty (x_i)^{1/q}t^i, \]
as well as $\tau=\sigma^{-1}:\tate\to \tate$ which both will be applied on matrices entry-wise.

\begin{thm} (See proof of \cite[Thm.~5.2.2]{mp:tdadmaicl})  \label{thm:conseq-of-abp}
Let $\Phi\in \Mat_{r\times r}(\bar{K}[t])$ be a matrix with
determinant $\det(\Phi)=c(t-\theta)^s$ for some $c\in \bar{K}^\times$ and $s\geq 1$.
If $\Psi\in \GL_r(\TT)\cap \Mat_{r\times r}(\mathbb{E})$ is a matrix such that
\[  \sigma(\Psi)=\Psi\Phi , \]
and $\Psi(\theta)$ the matrix obtained by evaluating all entries of $\Psi$ at  $t=\theta$,
then all algebraic relations over $\bar{K}$ between the entries of $\Psi(\theta)$ are specializations at  $t=\theta$ of the algebraic relations over $\bar{K}[t]$ between the entries of $\Psi$. 
\end{thm}

Actually, the matrix $\Phi$ often occurs as a matrix which represents the $\sigma$-action on a dual $t$-motive with respect to some $\bar{K}[t]$-basis, and $\Psi$ is the corresponding rigid analytic trivialization.

In \cite[Prop.~3.1.3]{ga-wb-mp:darasgvpc}, the authors showed that the condition $\Psi\in \Mat_{r\times r}(\mathbb{E})$ is already implied by the other conditions. Furthermore, the determinant condition $\det(\Phi)=c(t-\theta)^s$ for some $c\in \bar{K}^\times$ and $s\geq 1$ implies that also $\Psi(\theta)$ is invertible,
and in particular that $\Psi^{-1}$ does not have poles at $t=\theta$.

Our main theorem roughly states that all periods of uniformizable abelian $t$-modules (resp.~their coordinates) can be obtained as the entries of $\Psi^{-1}(\theta)=\Psi(\theta)^{-1}$ of a related dual $t$-motive. In combination with the theorem stated above, this enables one to study algebraic independence of periods by
studying the algebraic independence of the corresponding functions in the rigid analytic trivialization which is often
more accessible.

For example, the proof in \cite[Sect.~8]{am:ptmaip} that the coordinates of a fundamental period of the $n$-th tensor power of the Carlitz module are algebraically independent over $\bar{K}$ uses a special case of the strategy given here (comp.~Example~\ref{ex:carlitz-tensor-power}). 

This example also shows that the last step in the chain of contructions outlined in \cite[\S 4.5]{db-mp:ridmtt}
is not correct for $t$-modules in general, but only for certain kinds of $t$-modules like Drinfeld modules.

\medskip

For stating our main theorem more precisely, we need some more preparation.
The \emph{hyperdifferential operators} (also called iterative higher derivations) with respect to the variable~$t$ are the family of $\CC_\infty$-linear maps $(\hde{n})_{n\geq 0}$  on $\CC_\infty\ls{t}$ given by
\[    \hd{n}{\sum_{i=i_0}^\infty x_it^i} =\sum_{i=i_0}^\infty \binom{i}{n} x_it^{i-n} \]
for Laurent series $\sum_{i=i_0}^\infty x_it^i\in \laurent$, 
where $\binom{i}{n}\in \FF_p\subset \FF_q$ is the residue of the usual binomial coefficient.

For a square matrix $\Theta\in \Mat_{r\times r}(\laurent)$, we then define its $n$-th \emph{prolongation} ($n\geq 0$) to be
the $r(n+1)\times r(n+1)$-matrix given in block matrix form as
\begin{equation}\label{eq:rho_n-matrix} \rho_{[n]}(\Theta):= \begin{pmatrix}
\Theta & \hd{1}{\Theta}  & \hd{2}{\Theta} & \cdots& \hd{n}{\Theta} \\
0 & \Theta &  \hd{1}{\Theta} &  \ddots   & \vdots \\ 
\vdots &\ddots  & \ddots & \ddots &   \hd{2}{\Theta}\\
\vdots & & \ddots  &  \Theta &  \hd{1}{\Theta} \\
0 &  \cdots & \cdots & 0 &  \Theta
\end{pmatrix}, 
\end{equation}
where $\hd{1}{\Theta}$ etc. is
the matrix where we apply the hyperdifferential operators entry-wise. For more information on properties of 
these prolongations see \cite{am:ptmaip}. In this article, we just need that if $\Phi$ and $\Psi$ meet the conditions of Thm.~\ref{thm:conseq-of-abp}, then for all $n\geq 1$ also $ \rho_{[n]}(\Phi)$ and $\rho_{[n]}(\Psi)$ meet the conditions.
In particular, $\rho_{[n]}(\Psi)$ can be regarded as a rigid analytic trivialization of a dual $t$-motive.

However, apart from Section \ref{sec:dual-t-motive}, we will not deal with dual $t$-motives, but only recognize certain matrices satisfying the conditions in Thm.~\ref{thm:conseq-of-abp}. 
For $t$-motives, and also for dual $t$-motives in Section \ref{sec:dual-t-motive}, we use the terminologies of Brownawell-Papanikolas in \cite{db-mp:ridmtt} (see also \cite{am:ptmaip}) which are more general than those in \cite{ga:tm} and \cite{ga-wb-mp:darasgvpc}, respectively, since finite generation as $K[t]$-modules is not assumed. For perfect coefficient fields this more general notion of a $t$-motive already appears in Goss' book (cf.~\cite[Def.~5.4.2]{dg:bsffa}).

\begin{mainthm}[see Prop.~\ref{prop:def-of-r} and Thm.~\ref{thm:periods-as-special-values}]
Let $E$ be a uniformizable abelian $t$-module over $K$ of dimension $d$ and rank $r$. Let $\mot$ 
be the $t$-motive  associated to $E$ with $K[t]$-basis $\{m_1,\ldots, m_r\}$, let $\Theta$ be the matrix representing the $\tau$-action on $\mot$, i.e.~$\tau(m_i)=\sum_{j=1}^r \Theta_{ij} m_j$ for all $i=1,\ldots, r$, 
and let $\Upsilon\in \GL_r(\tate)$
be a rigid analytic trivialization for $\mot$ with respect to $\{m_1,\ldots, m_r\}$, i.e.~
\[  \tau\Bigl(\Upsilon\cdot \svect{m}{r} \Bigr)=\Upsilon\cdot \svect{m}{r}. \]

Then for any $B\in \GL_r(K[t])$, the matrices $\tilde{\Theta}=B\Theta\sigma(B)^{-1}\in \Mat_{r\times r}(\bar{K}[t])$ and 
$R:=\tau(\Upsilon)B^{-1}\in \GL_r(\tate)$, as well as the matrices $\rho_{[d-1]}(\tilde{\Theta})$ and $\rho_{[d-1]}(R)$ satisfy the conditions of $\Phi$ and $\Psi$ in Thm.~\ref{thm:conseq-of-abp}. Furthermore, there is a coordinate system of $E$ with corresponding isomorphism $E\cong \GG_{a}^d$, and
$B\in \GL_r(K[t])$ such that the coordinates of a basis of the period lattice $\Lambda$ are the values at $t=\theta$ 
of certain entries of the matrix $\rho_{[d-1]}( R^{-1})=\rho_{[d-1]}( R)^{-1}$.
\end{mainthm}

Our proof is even more explicit and shows which entries of $\rho_{[d-1]}(R)^{-1}$ are really used, e.g.~in some  examples one can even replace $\rho_{[d-1]}(R)^{-1}$ by $\rho_{[\alpha-1]}(R)^{-1}$ for some $\alpha<d$.

One step in the proof is an isomorphism of $\FF_q[t]$-modules between the lattice $\Lambda$ and the set of $\tau$-equivariant homomorphisms $\Hom_{K[t]}^\tau(\mot, \tate)$ via kind of generating series in one direction and taking residues at $t=\theta$ in the other direction (see Section \ref{sec:H}).
In the case of pure uniformizable $t$-modules this isomorphism is nothing else than the isomorphism induced by the pairing
\[ \begin{array}{rcll}
   \Lambda & \otimes_{\FF_q} &(\mot\otimes_{K[t]}\tate)^\tau &\longrightarrow (\tate)^\tau=\FF_q[t] \\
\lambda & \otimes & m &\longmapsto -\sum_{j=0}^\infty m\Bigl(\e\bigl((\dphi_t)^{-j-1}(\lambda)\bigr)\Bigr) t^j  
\end{array} \]
given in Anderson's paper \cite[\S 3]{ga:tm}, and in this case, Anderson also showed how to get the periods via higher residues at $t=\theta$. 

Our isomorphism, however, is much more general, since it is valid for any $t$-module, even for non-abelian ones.

\medskip

The isomorphism is actually a composition of two isomorphisms, and the connecting object is a certain $\FF_q[t]$-submodule $H$ of $E(\tate)$
(see Definition \ref{def:H-and-hat-H}). This object $H$ plays a very central role and is connecting various other
objects via natural isomorphisms, too.

\[ \xymatrix{
 & \{ (e_i)_{i\geq 0}\in T_t(E) \mid \lim_{i \to \infty}\norm{e_i}=0\} 
 \ar@{<-}[d]^{\cong \text{ (Prop.~\ref{prop:hat-H-isom-tate-module})}\eqref{item:H-isom-submodule-of-tate-module}} &  \\
 \Lambda \ar[r]^{\delta}_{\cong \text{ (Thm.~\ref{thm:iso-delta})}} &  
 H \ar[r]^(.4){\iota}_(.4){\cong \text{ (Thm.~\ref{thm:h-isom-M-tate-dual})}}
 & \Hom_{K[t]}^\tau(\mot, \tate) \\
&  (\dumot \otimes_{\CC_\infty[t]} \tate)^\sigma \ar@{<-}[u]_(.4){\cong \text{ (Thm.~\ref{thm:H-isom-dual-m-tate})}} & 
}\] 

Here $\dumot$ denotes the dual $t$-motive over $\CC_\infty$ associated to the $t$-module $E$
(see \cite[Sect.~4.4]{db-mp:ridmtt}), and $T_t(E)$ is the $t$-adic Tate module of $E$.
All these isomorphisms hold for arbitrary $t$-modules, even for non-abelian and non-$t$-finite ones.

Some compositions of the isomorphisms given here are already present in \cite[Sect.~5]{uh-akj:pthshcff}, due to
unpublished work of Anderson. For example, the isomorphism
\[ \Lambda\longrightarrow \{ (e_i)_{i\geq 0}\in T_t(E) \mid \lim_{i \to \infty}\norm{e_i}=0\} \]
can be obtained as a special case of the canonical bijection constructed between various objects given in Thm.~5.17 ibid. (see Remark \ref{rem:lambda-isom-H-in-HJ}), and in the case that $E$ is $t$-finite, the isomorphism
\[  (\dumot \otimes_{\CC_\infty[t]} \tate)^\sigma \longrightarrow \{ (e_i)_{i\geq 0}\in T_t(E) \mid \lim_{i \to \infty}\norm{e_i}=0\} \]
is a special case of the bijection constructed in Thm.~5.18 ibid. (see Remark \ref{rem:dual-motive-to-H-in-HJ}). 

Our isomorphisms, however, are obtained in a natural way, e.g.~as connecting homomorphisms using the snake lemma, and the isomorphism $H\to  \Hom_{K[t]}^\tau(\mot, \tate)$ seems to have not been
discovered for arbitrary $t$-modules, yet.

\smallskip

In the special case of Drinfeld modules, the object $H$ is already widely in use.
Namely, in this case, $H\subseteq E(\tate)=\tate$ equals (by definition) the solution space ${\rm Sol}(\Delta)\subseteq \tate$ of the difference operator 
 $\Delta=(\theta-t)+a_1\tau+\ldots + a_r\tau^r\in K[t]\{\tau\}$
 where $\phi_t=\theta+a_1\tau+\ldots + a_r\tau^r\in K\{\tau\}$ describes the $t$-action on $E$, and
 Pellarin showed in \cite[\S 4.2]{fp:aiacnn} that the Anderson generating functions are elements of ${\rm Sol}(\Delta)$. El-Guindy and Papanikolas \cite[Rem.~6.3]{aeg-mp:iagfdm} deduced that they even generate ${\rm Sol}(\Delta)$.
 The isomorphism $\delta$ recovers this relation between the Anderson generating functions and the elements of the lattice (see Remark \ref{rem:delta-for-drinfeld-modules} and Remark \ref{rem:tau-difference-solutions}). Using these Anderson generating functions, Pellarin ibid. also provided a rigid analytic trivialization of its dual $t$-motive explicitly  which equals the isomorphism of Thm.~\ref{thm:H-isom-dual-m-tate} in this case.
The relation of $H={\rm Sol}(\Delta)$ to the $t$-adic Tate-module is investigated in 
\cite[Sect.~3.2]{cc-mp:aipldm} where they also describe the action of the absolute Galois group on $T_t(E)$ using the Anderson generating functions.
 
The submodule ${\rm Sol}(\Delta)=H$ is also used at other places, e.g.~for studying periods, quasi-periods and logarithms \cite{aeg-mp:iagfdm}, vectorial Drinfeld modular forms \cite{fp-rp:vdmfta}
 or Drinfeld modules over Tate algebras \cite[Def.~6.4]{ba-fp-ftr:apclsvta}.

\medskip

In the proof of our main theorem, a second step is how one can recover the higher residues which provide the coordinates of the periods as the values at $t=\theta$ of certain entries of this particular matrix $\rho_{[d-1]}(R^{-1})$.
This is the content of Section~\ref{sec:matrix-for-specialization}.

\subsection*{Acknowledgement}

I would like to thank Chieh-Yu Chang as well as the referee for helpful comments for improving the paper.
I would also like to thank Quentin Gazda for interesting discussions on certain aspects of the topic.

\section{Generalities}\label{sec:generalities}

\subsection{Base rings and operators}

Let $\FF_q$ be the finite field with $q$ elements and characteristic $p$, and $K$ a finite extension of the rational function field $\FF_q(\theta)$ in the variable $\theta$. We choose an extension to $K$ of the absolute value $\betr{\cdot }$ which is given on $\FF_q(\theta)$  by $\betr{\theta}=q$.
 Furthermore, $K_\infty\supseteq \FF_q\ls{\frac{1}{\theta}}$ denotes the completion of $K$ at this infinite place, and
$\CC_\infty$ the completion of an algebraic closure of $K_\infty$. Furthermore, let $\bar{K}$ be the algebraic closure of $K$ inside $\CC_\infty$.

All the commutative rings occuring will be subrings of the field of Laurent series $\laurent$, like
the polynomial rings $K[t]$ and $\bar{K}[t]$, the power series ring $\powser$ and the
Tate algebra $\tate$, i.e.~the algebra of series which are convergent for $\betr{t}\leq 1$. 

On $\laurent$ we have several operations which will induce operations on (most of) these subrings.

First at all, there is the twisting $\tau:\laurent\to \laurent$ given by 
\[  \tau(f) :=\sum_{i=i_0}^\infty (x_i)^qt^i \]
for $f=\sum_{i=i_0}^\infty x_it^i\in \laurent$, and the inverse twisting $\sigma:\laurent\to \laurent$ given by 
\[  \sigma(f) := \sum_{i=i_0}^\infty (x_i)^{1/q}t^i. \]

Furthermore, we have an action of the hyperdifferential operators with respect to~$t$, i.e.~the sequence of $\CC_\infty$-linear maps $(\hde{n})_{n\geq 0}$ given by
\[    \hd{n}{\sum_{i=i_0}^\infty x_it^i} =\sum_{i=i_0}^\infty \binom{i}{n} x_it^{i-n}, \]
where $\binom{i}{n}\in \FF_p\subset \FF_q$ is the residue of the usual binomial coefficient.
The image $\hd{n}{f}$ of some $f\in \laurent$ is called the $n$-th hyperderivative of $f$.

{The reader should be aware that the hyperderivatives we use here are different from those commonly used in this area of research (e.g.~in \cite{db:lidddm-I}, \cite{db:meagli}, \cite{db-ld:lidddm-II}), since the latter are obtained by using the hyperdifferential operators with respect to $\theta$ on $\FF_q(\theta)$ and its separable extensions.}

\medskip

While the twisting $\tau$ and the hyperdifferential operators restrict to endomorphisms on all subrings of $\laurent$ which occur in this paper, the inverse twisting $\sigma$ is only defined for perfect coefficient fields, in particular not on $K[t]$, but on $\bar{K}[t]$.
It is also obvious that the hyperdifferential operators commute with the twistings $\tau$ and~$\sigma$.

When we apply the twisting operators $\tau$ and $\sigma$ as well as the hyperdifferential operators to matrices it is meant that we apply them coefficient-wise.

\medskip

As the hyperdifferential operators commute with twisting, the operation of taking the $n$-th prolongation $\rho_{[n]}(\Theta)$ of square matrices $\Theta\in \Mat_{r\times r}(\laurent)$ (see Equation \eqref{eq:rho_n-matrix}) also commutes with twisting.
Furthermore, one can check that the $n$-th prolongation is a ring homomorphism
$$\rho_{[n]}:\Mat_{r\times r}(\laurent)\to \Mat_{r(n+1)\times r(n+1)}(\laurent)$$ (see~\cite{am-rp:iddbcppte} and \cite{am:ptmaip}).

\subsection{Basic objects and properties}\label{subsec:basic-objects}

We now define the basic objects and recall their main properties  used in this paper. For further details, we refer the reader to \cite{db-mp:ridmtt} or \cite{am:ptmaip}.

\medskip

$(E,\phi)$ denotes a $t$-module of dimension $d$ over the field $K$ with generic characteristic $\ell:\FF_q[t]\to K,t\mapsto \theta$. Hence, $E$ is isomorphic to $\GG_{a,K}^d$ as an algebraic group over $K$, and $\phi$ is a homomorphism
\[  \phi:\FF_q[t]\to \End_{\mathrm{grp},\FF_q}(E), a\mapsto \phi_a\]
into the group of $\FF_q$-linear homomorphisms of algebraic groups over $K$ (also called homomorphisms of
$\FF_q$-module schemes). 
The induced action on the Lie algebra $\Lie(E)$ will be denoted by
\[ \dphi:\FF_q[t]\to \End_K(\Lie(E)), \]
and by hypothesis on $t$-modules, $N:=\dphi_t-\theta$ is a nilpotent endomorphism on $\Lie(E)$.

Associated to $E$, one has the exponential map
\[ \e:\Lie(E)(\CC_\infty) \to E(\CC_\infty) \]
satisfying $\e(\dphi_a(x))=\phi_a(\e(x))$ for all $a\in \FF_q[t]$ and $x\in \Lie(E)(\CC_\infty)$.
The kernel of the exponential $\Ker(\e)$ is an $\FF_q[t]$-sub\-module of $\Lie(E)(\CC_\infty)$, and is called the period lattice $\Lambda_E:=\Ker(\e)$. Since, the $t$-module $E$ will be fixed throughout the paper, we will usually omit the subscript $E$ and simply write $\Lambda$ instead of $\Lambda_E$.

On $\Lie(E)(\CC_\infty)$ and on $E(\CC_\infty)$ we fix a norm $\norm{x}=\max \{ \betr{x_j}\, \mid j=1,\ldots,d\}$
where $x$ corresponds to $\transp{(x_1,\ldots, x_d)}$ via some  fixed choice of coordinates $E(\CC_\infty)\cong \GG_a^d(\CC_\infty)=\CC_\infty^{d}$, and induced isomorphism $\Lie(E)(\CC_\infty)\cong \CC_\infty^{d}$.

The exponential $\e$ is then a local isometry (cf.~\cite[Lemma 5.3]{uh-akj:pthshcff}). This means that there exists $\eps>0$ such that $\e$ restricts to a bijection
\[  \e:B_{\Lie(E)}(0,\eps) \to B_{E}(0,\eps) \]
satisfying $\norm{\e(x)}=\norm{x}$ for all $x\in B_{\Lie(E)}(0,\eps)$, where
\begin{eqnarray*}
B_{\Lie(E)}(0,\eps)  &=& \{ x\in \Lie(E)(\CC_\infty) \mid \norm{x}<\eps \},\\
B_{E}(0,\eps)  &=& \{ x\in E(\CC_\infty) \mid \norm{x}<\eps \}.
\end{eqnarray*}

For our main theorem, we will also assume in Section \ref{sec:matrix-for-specialization} that $E$ is abelian and uniformizable,  
but Sections \ref{sec:H} and \ref{sec:dual-t-motive} are valid for arbitrary $t$-modules $E$.

\smallskip

We emphasize a fact on $\Lambda$ which is usually stated for abelian $t$-modules.

\begin{defn}
We define the $t$-rank of $E$, $\trk(E)$ to be the dimension of the $t$-torsion 
$E[\phi_t]=\{ e\in E(\CC_\infty)\mid \phi_t(e)=0 \}$ as an $\FF_q$-vector space.
\end{defn}

\begin{prop}\label{prop:lambda-fin-gen}
For any $t$-module $E$, the following hold.
\begin{enumerate}
\item $\trk(E)\in \NN$ is well-defined, i.e.~$E[\phi_t]$ is a finite dimensional $\FF_q$-vector space.
\item The period lattice $\Lambda$ is discrete in $\Lie(E)(\CC_\infty)$, and it is a free $\FF_q[t]$-module of rank
not exceeding $\trk(E)$.
\end{enumerate}
\end{prop}

\begin{proof}
\begin{enumerate}
\item $\phi_t:E\to E$ is a morphism of algebraic groups (even of $\FF_q$-module schemes) which induces an isomorphism on the Lie-algebra $\dphi_t:
\Lie(E)(K)\to \Lie(E)(K)$, since $\det(\dphi_t)=\theta^d\in K^\times$. Hence, $\phi_t$ is a finite \'etale covering of $\FF_q$-module schemes, and its kernel is a finite closed reduced $\FF_q$-module subscheme of $E$. Therefore,
its $\CC_\infty$-points $E[\phi_t]$ build an $\FF_q$-vector space of finite dimension.
\item Discreteness of $\Lambda$ follows from $\e$ being a local isometry. That $\Lambda$ is free of rank $\rank(\Lambda)\leq\trk(E)$ is exactly
what Anderson proved in \cite[Lemma 2.4.1]{ga:tm}, and abelianness of the $t$-module only came into play to relate $\trk(E)$ to the rank of the associated $t$-motive. \qedhere
\end{enumerate}
\end{proof}

\medskip

The $t$-motive associated to $E$ is the abelian group of $\FF_q$-linear homomorphisms of algebraic groups
over $K$ 
\[ \mot:=\mot(E)=\Hom_{\mathrm{grp},\FF_q}(E,\GG_{a,K}) \]
with $\FF_q[t]$-action via $a\cdot m:=m\circ \phi_a$ for all $m\in \mot$ and $a\in \FF_q[t]$,
and with $K\{\tau\}=\End_{\mathrm{grp},\FF_q}(\GG_{a,K})$-action via $\psi\cdot m:=\psi\circ m$ for all $m\in \mot$ and $\psi\in K\{\tau\}$. This is a $t$-motive in the terminology of \cite{db-mp:ridmtt} (see also \cite{am:ptmaip} for a precise definition).
When $E$ is abelian, then $\mot$ is finitely generated and free of some rank $r$ as a $K[t]$-module. By \cite[Prop.~1.8.3]{ga:tm} this rank equals the $t$-rank $\trk(E)$ that we defined above.
In the abelian case, we fix a $K[t]$-basis
$\{m_1,\ldots,m_r\}$ of $\mot$. With respect to this basis the $\tau$-action can be described as 
\[  \tau\svect{m}{r} = \Theta \svect{m}{r} \]
for some $\Theta\in \Mat_{r\times r}(K[t])$, and its determinant $\det(\Theta)$ equals $c(t-\theta)^s$ for some
$c\in K^\times$ and $s\geq 1$.

\subsection{Further objects}\label{subsec:further-objects}

We will need some more objects:
\begin{align*}
E\ps{t} &:= \left\{ \sum_{i=0}^\infty e_i t^i \bigmid e_i\in E(\CC_\infty)\right\}
\qquad \text{formal power series with coefficients in }E(\CC_\infty),\\
E[t] &:=  \left\{ \sum_{i\geq 0} e_i t^i\in E\ps{t} \bigmid e_i=0 \text{ for }i\gg 0\right\}
\cong E(\CC_\infty) \otimes_{\FF_q} \FF_q[t], 
\\
E\cs{t} &:= \left\{ \sum_{i\geq 0} e_i t^i \in E\ps{t}\bigmid \lim_{i\to \infty}\norm{e_i} =0 \right\}.
\end{align*}
Similarly, we define $\Lie(E)\ps{t}$, $\Lie(E)[t]$ and $\Lie(E)\cs{t}$, as well as $\Lambda\ps{t}$, $\Lambda[t]$ and $\Lambda\cs{t}$. However, since $\Lambda$ is discrete, $\Lambda\cs{t}=\Lambda[t]$.

\medskip

We should remark that on these objects, we have two different actions of $t$, namely one on the coefficients via $\phi_t$ or $\dphi_t$, respectively, and the other by raising the power of $t$ in the series. By abuse of notation, we will also
denote the action on the coefficients by $\phi_t$ and $\dphi_t$:
\[  \phi_t( \sum_{i\geq 0} e_i t^i ):=  \sum_{i\geq 0} \phi_t(e_i) t^i \quad\text{and}\quad \dphi_t( \sum_{i\geq 0} x_i t^i ):=  \sum_{i\geq 0} \dphi_t(x_i) t^i \]
for $\sum_{i\geq 0} e_i t^i \in E\ps{t}$ and $\sum_{i\geq 0} x_i t^i\in \Lie(E)\ps{t}$.

\begin{rem}
\begin{enumerate}
\item In the case that $E$ is a Drinfeld module, $E\cs{t}$ is a special case of a Drinfeld module 
over the Tate-algebra $\tate$ as defined in \cite{ba-fp-ftr:apclsvta}.
\item Although for different choices of coordinates $E(\CC_\infty)\cong \GG_a^d(\CC_\infty)=\CC_\infty^{d}$, the norms
$\norm{\cdot}$ might not be equivalent, the sets of sequences converging to $0$ are the same for each norm. Hence, the definitions of $E\cs{t}$ and $\Lie(E)\cs{t}$ are independent of the chosen coordinate system defining the norm.
\item Be aware that $E\ps{t} \supsetneq E(\CC_\infty)\otimes_{\FF_q} \FF_q\ps{t}$, since $E(\CC_\infty)$ is an infinite dimensional $\FF_q$-vector space.
\end{enumerate}
\end{rem}

\begin{lem}\label{lem:ses-exp-tate}
The exact sequence of $\FF_q[t]$-modules (via $\dphi$ and $\phi$, resp.)
\[0 \to \Lambda \to \Lie(E)(\CC_\infty) \xrightarrow{\e} E(\CC_\infty)\]
induces an exact sequence of $\FF_q[t]\otimes_{\FF_q} \FF_q[t]$-modules
\[ 0\to \Lambda[t]  \to \Lie(E)\cs{t} 
 \xrightarrow{\e\cs{t}} E\cs{t}, \]
by applying the homomorphisms coefficient-wise. 
\end{lem}

\begin{proof}
When we apply the maps of the exponential sequence coefficient-wise, we obviously obtain an exact sequence
\begin{equation}\label{eq:exp-seq-over-ps}
 0\to \Lambda\ps{t}  \to \Lie(E)\ps{t} 
 \xrightarrow{\e\ps{t}} E\ps{t}.
\end{equation}
Since $\e$ is a local isometry,
$\e\ps{t}(\sum_{i\geq 0} x_it^i)$ is in $E\cs{t}$ if and only if $\sum_{i\geq 0} x_it^i\in \Lie(E)\cs{t}$.
Taking into account that $\Lambda\cs{t} =\Lambda[t]$, the exact sequence \eqref{eq:exp-seq-over-ps} restricts to
an exact sequence
\[ 0\to \Lambda[t]  \to \Lie(E)\cs{t} 
 \xrightarrow{\e\cs{t}} E\cs{t} \]
where $\e\cs{t}$ is just the restriction of $\e\ps{t}$ to ${\Lie(E)\cs{t}}$.
\end{proof}

\section{The subset $H$ and natural isomorphisms}\label{sec:H}

In this section, we don't assume any additional property (like abelianess or uniformizability) on the $t$-module $E$, since all constructions and theorems are  valid without further assumptions. The central object of this part is the following subset $H$ of $E\cs{t}$.

\subsection{Definition of $H$ and first isomorphism}

\begin{defn}\label{def:H-and-hat-H}
The subset $H_E$ of $E\cs{t}$, as well as the subset $\hat{H}_E$ of $E\ps{t}$, consist of those elements on which both $t$-actions coincide.
\begin{eqnarray*}
\hat{H}_E &:=&  \{ h=\sum_i e_i t^i\in E\ps{t} \mid \phi_t(h)=h\cdot t \}\\
H_E &:=& \{ h=\sum_i e_i t^i\in E\cs{t} \mid \phi_t(h)=h\cdot t \} =\hat{H}_E\cap E\cs{t}.
\end{eqnarray*}
Since, we work with a fixed $t$-module $E$ throughout the paper, we will always omit the subscript $E$,
and simply write $\hat{H}$ and $H$ for $\hat{H}_E$ and $H_E$, respectively.
\end{defn}

\begin{prop}\label{prop:hat-H-isom-tate-module}
\begin{enumerate}
\item\label{item:descr-of-hat-H} $h=\sum_i e_i t^i\in E\ps{t}$ is an element of $\hat{H}$, if and only if
$(e_i)_{i\geq 0}$ is a compatible system of $t^{i+1}$-torsion, i.e.~$e_i\in E[\phi_{t^{i+1}}]$ and $\phi_t(e_{i+1})=e_i$  for all $i\geq 0$.
\item\label{item:action-on-hat-H} $\hat{H}$ carries a natural $\FF_q\ps{t}$-action given by
\[  h\cdot f= \left(\sum_i e_i t^i\right)\cdot (\sum_j a_jt^j)= \sum_{n\geq 0} (\sum_{k=0}^n  a_k e_{n-k}) t^n \]
for $h=\sum_i e_i t^i\in \hat{H}$ and $f=\sum_j a_jt^j\in \FF_q\ps{t}$.
\item\label{item:hat-H-isom-tate-module} Via this action, $\hat{H}$ is isomorphic as $\FF_q\ps{t}$-module to the $t$-adic Tate-module 
$T_t(E)=\varprojlim_{i} E[\phi_{t^{i+1}}]$ of $E$ via $h=\sum_i e_i t^i\mapsto (e_i)_{i\geq 0}$.
\item\label{item:H-isom-submodule-of-tate-module} $H$ is isomorphic to the $\FF_q[t]$-submodule of the Tate-module $T_t(E)$ of those compatible systems $(e_i)_{i\geq 0}$ which tend to zero for $i\to \infty$.
\end{enumerate}
\end{prop}

\begin{proof}
By definition, $h=\sum_{i\geq 0} e_i t^i\in E\ps{t}$ is in $\hat{H}$ if and only if
$ \phi_t(h)=\sum_{i\geq 0} \phi_t(e_i) t^i$ equals $h\cdot t=\sum_{i\geq 0}e_i t^{i+1}$. Comparing coefficients
this is equivalent to $\phi_t(e_0)=0$, and $\phi_t(e_{i+1})=e_i$ for all $i\geq 0$, whence the claim in \eqref{item:descr-of-hat-H}. For verifying \eqref{item:action-on-hat-H}, we recognize that the given formula for $h\in E\ps{t}$ is the usual action of formal power series, and one easily checks by straightforward computation that $\hat{H}$ is stable under this action. Items \eqref{item:hat-H-isom-tate-module} and \eqref{item:H-isom-submodule-of-tate-module} are then just consequences of that.
\end{proof}

\begin{rem}
For Drinfeld modules, the submodule $H\subseteq E\cs{t}\cong \tate$ has already been considered and used at several places as the solution space of the difference operator $\Delta:=\phi_t-t\in K[t]\{\tau\}$, or as ``$(\theta-t)$-torsion'' (cf.~Introduction). In this case, also the connection of $H$ to the $t$-adic Tate-module has already been investigated in \cite[Sect.~3]{cc-mp:aipldm}, including the Galois representation on the Tate-module.
\end{rem}

\begin{prop}
Let $r:=\trk(E)$.
\begin{enumerate}
\item $\hat{H}$ is a free $\FF_q\ps{t}$-module of rank $r$.
\item $H$ is a free $\FF_q[t]$-module of rank less or equal to $r$. If $\{h_1,\ldots, h_s\}$ is a basis of $H$, then
$\{h_1,\ldots, h_s\}$ is part of an  $\FF_q\ps{t}$-basis of $\hat{H}$. In particular, they are  $\FF_q\ps{t}$-linearly independent.
\end{enumerate}
\end{prop}

As we didn't find a proof for the finite rank of the Tate-module for arbitrary $t$-modules $E$, we give a proof here.

\begin{proof}
\begin{enumerate}
\item As we have seen in the proof of Prop.~\ref{prop:lambda-fin-gen}, $\phi_t:E\to E$ is a finite \'etale covering of degree $q^r$. Hence, for all $x\in E(\CC_\infty)$, the preimage $\phi^{-1}(\{x\})\subseteq E(\CC_\infty)$ has exactly $q^r$ elements. Therefore, for all $i\geq 0$, the $t^{i+1}$-torsion $E[\phi_{t^{i+1}}]$ is a free $\FF_q[t]/(t^{i+1})$-module of rank $r$. Passing to the limit, gives the desired conclusion.
\item In Theorem \ref{thm:iso-delta}, we will see that $H$ is isomorphic to the period lattice $\Lambda$. Hence, by Prop.~\ref{prop:lambda-fin-gen}, it is a free $\FF_q[t]$-module of rank $\rank(H)\leq \trk(E)$.

Let $h_1,\ldots, h_s$ be a basis of $H$. Since by Nakayama's lemma, a set $\{f_1,\ldots, f_r\}$ of elements of $\hat{H}$ is an $\FF_q\ps{t}$-basis of $\hat{H}$ if and only if its reductions modulo $t$ are an $\FF_q$-basis of
$\hat{H}/t\hat{H}=E[\phi_{t}]$, we only have to show that the constant terms $e_j:=h_j|_{t=0}\in E[\phi_t]$ are $\FF_q$-linearly independent.\\
For the contrary, assume that there is a non-trivial relation $\sum_{j=1}^s c_je_j=0$ with $c_1,\ldots, c_s\in \FF_q$. Then $\sum_{j=1}^s c_jh_j\in H\cap t\hat{H}=tH$. As $h_1,\ldots, h_s$ is a basis of $H$, there are $d_1,\ldots, d_s\in \FF_q[t]$ such that $\sum_{j=1}^s c_jh_j=t\cdot (\sum_{j=1}^s d_jh_j)$, i.e.~
\[  0=\sum_{j=1}^s c_jh_j - t\cdot (\sum_{j=1}^s d_jh_j)=\sum_{j=1}^s (c_j-td_j)h_j, \]
contradicting the assumption that $h_1,\ldots, h_s$ is linearly independent. \qedhere
\end{enumerate}
\end{proof}


\subsection{Second isomorphism $\Lambda\to H$}

The next step is to show that $H$ is isomorphic to the lattice $\Lambda$.

\begin{rem}\label{rem:lambda-isom-H-in-HJ}
Since, we already know that $H$ is isomorphic to the set of compatible systems $(e_i)_{i\geq 0}$ of $\phi_{t^i}$-torsion elements which tend to zero, the isomorphism $\Lambda\to H$ is also obtained as a special case of the canonical bijection given in \cite[Thm.~5.17]{uh-akj:pthshcff}. However, we will have a natural description of the homomorphism
$\Lambda\to H$ which already implies that it is a monomorphism, and only surjectivity is shown in the same lines
as in \cite{uh-akj:pthshcff}.
\end{rem}

The difference $\dphi_t-t$ of the two $t$-actions on $\Lie(E)\cs{t}$ is an isomorphism, since it is $\CC_\infty\cs{t}$-linear with determinant $(\theta-t)^d\in \CC_\infty\cs{t}^\times$.

Writing, $\phi_t-t$ for the difference of the two $t$-actions on $E\cs{t}$, we therefore have a commutative diagram of $\FF_q[t]\otimes_{\FF_q} \FF_q[t]$-modules with exact rows 

\centerline{
\xymatrix{
 0 \ar[r] & 0\ar[d] \ar[r] &  \Lie(E)\cs{t}  \ar[d]^{\e\cs{t}} \ar[r]^{\dphi_t-t} &  
 \Lie(E)\cs{t}\ar[d]^{\e\cs{t}}\ar[r]  & 0 \\
 0 \ar[r] & H  \ar[r]  & E\cs{t} \ar[r]^{\phi_t-t} & E\cs{t} .
}}

Inserting the kernels of the vertical maps (cf.~Lemma \ref{lem:ses-exp-tate}) and cokernel of the first vertical map, we get an exact  sequence by the snake lemma.

\centerline{
\xymatrix{
 & 0 \ar[r] \ar@{-->}[d]  &  \Lambda[t]  \ar[d]  \ar[r]^{\dphi_t-t} &  \Lambda[t] \ar[d] 
    \ar@{-->}'[r] `r[d]`[d]+/d 3ex/ `^d[lll]+/l 3ex/ `[dddll]      [dddll]
 & \\
 0 \ar[r] & 0\ar[d] \ar[r] &  \Lie(E)\cs{t}  \ar[d]^{\e\cs{t}} \ar[r]^{\dphi_t-t} &  
 \Lie(E)\cs{t}\ar[d]^{\e\cs{t}}\ar[r]  & 0 \\
 0 \ar[r] & H \ar@{=}[d] \ar[r]  & E\cs{t} \ar[r]^{\phi_t-t} & E\cs{t} &\\
 & H &&&
}}

\begin{thm}\label{thm:iso-delta}
The homomorphism  of $\FF_q[t]\otimes_{\FF_q} \FF_q[t]$-modules $\Lambda[t]\to H$ given by the snake lemma
induces an isomorphism of $\FF_q[t]$-modules 
\[ \delta:\Lambda\to H, \lambda\mapsto \e\cs{t}\left( (\dphi_t-t)^{-1}(\lambda)\right)
=\sum_{i\geq 0} \e\left(\dphi_t^{-i-1}(\lambda)\right)\cdot t^i.\]
\end{thm}

\begin{proof}
Recall that by definition, the left and right $\FF_q[t]$-actions on $H$ coincide. Hence, we are free to use the left or right $\FF_q[t]$-action from $E\cs{t}$ whichever is more suitable to the task.  
The cokernel of $\dphi_t-t:\Lambda[t]\to \Lambda[t]$ is isomorphic to $\Lambda=\Lambda\cdot t^0$ with $t$-action via $\dphi_t$. Hence, we get an induced injective homomorphism of $\FF_q[t]$-modules $\Lambda\to H$.
Following the arrows in the diagram, we see that indeed $\delta$ is given by 
\[ \delta(\lambda)=\e\cs{t}\left( (\dphi_t-t)^{-1}(\lambda)\right). \]
The second formula for $\delta(\lambda)$ is then obtained by recognizing that $(\dphi_t-t)^{-1}(\lambda)$
is given by the geometric series $\sum_{i\geq 0} \dphi_t^{-i-1}(\lambda)\cdot t^i\in \Lie(E)\ps{t}$.
It remains to show that $\delta$ is surjective.

First remark, that there exists $0<\eps_0<\eps$ such that
\[ \dphi_t\bigl(B_{\Lie(E)}(0,\eps_0)\bigr)\subseteq B_{\Lie(E)}(0,\eps), \]
since $\dphi_t:\Lie(E)(\CC_\infty)\to\Lie(E)(\CC_\infty)$ is $\CC_\infty$-linear, and hence continuous.

Now, let $h=\sum_i e_i t^i\in H$. By definition $\lim_{i \to \infty}\norm{e_i}=0$, hence
there is $n\in \NN$ such that $e_i\in B_E(0,\eps_0)\subseteq E(\CC_\infty)$ for all $i\geq n$. Therefore, for any $i\geq n$ there is a unique $\lambda_i\in B_{\Lie(E)}(0,\eps_0)\subseteq \Lie(E)(\CC_\infty)$ such that $\e(\lambda_i)=e_i$,
and we define $\lambda:=\dphi_{t^{i+1}}(\lambda_i)\in  \Lie(E)(\CC_\infty)$. This definition is independent of the chosen $i\geq n$, since 
\[ \e\left( \dphi_{t}(\lambda_{i+1})-\lambda_i\right) = \phi_t(\e(\lambda_{i+1}))-\e(\lambda_i)
=\phi_t(e_{i+1})-e_i=0, \]
by Prop.~\ref{prop:hat-H-isom-tate-module}\eqref{item:descr-of-hat-H},
i.e.~$\dphi_{t}(\lambda_{i+1})-\lambda_i\in \Lambda\cap B_{\Lie(E)}(0,\eps)=\{0\}$.

Furthermore, $\lambda$ is indeed in $\Lambda$, since
\[ \e(\lambda)=\e\left(\dphi_{t^{i+1}}(\lambda_i)\right)=\phi_{t^{i+1}}(e_{i})=0.\]

Finally, we have
\begin{eqnarray*}
 \delta (\lambda) &=& \sum_{i\geq 0} \e\left(\dphi_t^{-i-1}(\lambda)\right)\cdot t^i \\
&=& \sum_{0\leq i< n} \e\left(\dphi_t^{n-i}(\lambda_{n})\right)\cdot t^i
+\sum_{i\geq n}  \e\left( \lambda_i\right) \cdot t^i \\
&=&  \sum_{0\leq i< n} \phi_t^{n-i}(e_{n}) t^i
+ \sum_{i\geq n} e_i t^i \\
&=&  \sum_{0\leq i< n} e_{i} t^i + \sum_{i\geq n} e_i  t^i =h. 
\end{eqnarray*}
\end{proof}

\begin{rem}\label{rem:delta-for-drinfeld-modules}
In the case of $E$ being a Drinfeld module, $\dphi_t$ is just multiplication by~$\theta$.
Hence,
 the isomorphism $\delta:\Lambda\to
H\subseteq \tate$ above associates to a period $\lambda\in \Lambda$ its Anderson generating function 
\[ g_{\phi}(\lambda;t)=
\sum_{i\geq 0} \e\left(\frac{\lambda}{\theta^{i+1}}\right)\cdot t^i\]
as given for example in \cite{fp:aiacnn} or \cite{aeg-mp:iagfdm}.
The fact that one recovers $\lambda$ from $g_{\phi}(\lambda;t)$ via
\[  \lambda=-\res(g_{\phi}(\lambda;t)\, dt) \]
is also valid in the general setting, as the following proposition shows.
\end{rem}

\begin{prop}\label{prop:inverse-of-delta-explicit}
The inverse of the isomorphism
\[\delta:\Lambda\to H, \lambda\mapsto \e\cs{t}\left( (\dphi_t-t)^{-1}(\lambda)\right)\]
is explicitly given by sending $h\in H$ to the residue at $t=\theta$ of $-h\, dt$; considered coordinate-wise with respect to some coordinate system of $E$, and corresponding coordinate system of $\Lie(E)$,
\[ -\res: H\to \Lambda, h\mapsto -\res(h\, dt) .
\]
\end{prop}

\begin{proof}
For the proof, we fix a coordinate system $E(\CC_\infty)\cong \GG_a^d(\CC_\infty)\cong \CC_\infty^d$ and corresponding coordinate system $\Lie(E)(\CC_\infty)\cong \CC_\infty^d$, and compute in these coordinates without mentioning it explicitly.
By definition of a $t$-module, $\dphi_t=\theta+N$ for a $\CC_\infty$-linear nilpotent operator $N$ on $\Lie(E)(\CC_\infty)$. Since, $\dim_{\CC_\infty}(\Lie(E)(\CC_\infty))=d$, we have $N^d=0$. Hence in 
$\End_{\tate}(\Lie(E)\cs{t})\cong \End_{\tate}(\tate^d)$, we obtain
\[ (\dphi_t-t)^{-1}=(\theta+N-t)^{-1}=-(t-\theta)^{-1}(1-(t-\theta)^{-1}N)^{-1}
=- \sum_{k=0}^{d-1} (t-\theta)^{-k-1}N^k. \]
Further denote by $N^{(j)}=\tau^j(N)$ the operator with twisted coefficients, then for $j\geq 1$:
\begin{eqnarray*}
\left( (\dphi_t-t)^{-1}\right)^{(j)} &=& (\theta^{q^j}+N^{(j)}-t)^{-1}\\
&=& (\theta^{q^j}+N^{(j)}-\theta)^{-1}\cdot \left( 1-(t-\theta)(\theta^{q^j}+N^{(j)}-\theta)^{-1}\right)^{-1} \\
&=& \sum_{k=0}^\infty (\theta^{q^j}+N^{(j)}-\theta)^{-k-1}(t-\theta)^k\quad \in \End(\CC_\infty^d)[[t-\theta]]
\end{eqnarray*}

Let $\e(x)=\sum_{j=0}^\infty e_{\phi,j}\tau^j(x)=x+\sum_{j=1}^\infty e_{\phi,j}\tau^j(x)$, then
\begin{eqnarray*}
 && \e\cs{t}\bigl( (\dphi_t-t)^{-1}(\lambda)\bigr)
=  (\dphi_t-t)^{-1}(\lambda) +  \sum_{j=1}^\infty e_{\phi,j} \left( (\dphi_t-t)^{-1}\right)^{(j)}(\lambda^{q^j}) \\
&=& - \sum_{k=0}^{d-1} (t-\theta)^{-k-1}N^k(\lambda)+
\sum_{j=1}^\infty e_{\phi,j} \sum_{k=0}^\infty (\theta^{q^j}+N^{(j)}-\theta)^{-k-1}(\lambda^{q^j})(t-\theta)^k\\
&=&  - \sum_{k=0}^{d-1} (t-\theta)^{-k-1}N^k(\lambda)+
 \sum_{k=0}^\infty \sum_{j=1}^\infty e_{\phi,j}  (\theta^{q^j}+N^{(j)}-\theta)^{-k-1}(\lambda^{q^j})\cdot (t-\theta)^k 
\end{eqnarray*}
Hence, the coefficient of $(t-\theta)^{-1}$ is $-N^0(\lambda)=-\lambda$.
\end{proof}

\subsection{Third isomorphism $H\to  \Hom_{K[t]}^\tau( \mot(E), \tate)$}

\medskip

In the following we show that $H$ is also isomorphic to $\Hom_{K[t]}^\tau( \mot, \tate)$, where as defined earlier $\mot=\Hom_{\mathrm{grp},\FF_q}(E,\GG_{a,K}) $ is the $t$-motive associated to~$E$.

The natural homomorphism of $\FF_q$-vector spaces
\begin{equation}\label{eq:iso-E-to-M-dual}
     E(\CC_\infty) \longrightarrow \Hom_{K}^{\tau}(\mot, \CC_\infty), e\mapsto
\left\{ \mu_e: m\mapsto m(e) \right\}
\end{equation}
is an isomorphism, since after a choice of coordinate system $E(\CC_\infty)\cong \CC_\infty^d$ the latter is isomorphic to the bidual vector space 
$$E(\CC_\infty)^{\vee\vee}
=\Hom_{\CC_\infty}(\Hom_{\CC_\infty}(E(\CC_\infty),\CC_\infty),\CC_\infty)$$ of $E(\CC_\infty)$.
The homomorphism \eqref{eq:iso-E-to-M-dual} is even compatible with the $t$-action on $E$ via $\phi_t$ and the $t$-action on 
$\mu\in \Hom_{K}^{\tau}(\mot, \CC_\infty)$ via $(t\cdot \mu)(m)=
\mu(m\circ \phi_t)$ for all $m\in \mot$.

\begin{thm}\label{thm:h-isom-M-tate-dual}
There is a natural isomorphism of $\FF_q\ps{t}$-modules
\[  \hat{H} \longrightarrow \Hom_{K[t]}^\tau( \mot, \powser) \]
which restricts to an isomorphism of $\FF_q[t]$-modules
\[ \iota: H\longrightarrow \Hom_{K[t]}^\tau( \mot, \tate). \]
\end{thm}

\begin{proof}
The isomorphism \eqref{eq:iso-E-to-M-dual} induces an
isomorphism of $\FF_q[t]\otimes_{\FF_q} \FF_q\ps{t}$-modules
\begin{eqnarray*}
 E\ps{t} &\longrightarrow & \Hom_{K}^{\tau}(\mot, \CC_\infty)\ps{t}=\Hom_{K}^{\tau}(\mot, \CC_\infty\ps{t})\\
 \sum_i e_i t^i &\mapsto & \left\{
 \sum_i \mu_{e_i} t^i  : m\mapsto \sum_i m(e_i)t^i \right\},
\end{eqnarray*}
By compatibility with the $t$-actions, the image of $\hat{H}\subseteq  E\ps{t}$ are exactly those
homomorphisms $\mu:\mot\to  \CC_\infty\ps{t}$ for which $\mu(m\circ \phi_t)=\mu(m)\cdot t$ for all $m\in \mot$, i.e.~the $K[t]$-linear ones, inducing the isomorphism
\[  \hat{H} \longrightarrow \Hom_{K[t]}^\tau( \mot, \powser). \]
Furthermore, one has $\lim\limits_{i\to \infty} \norm{e_i}=0$ if and only if for all $m\in \mot$,
$\lim\limits_{i\to \infty} \betr{m(e_i)}=0$. Hence, the given isomorphism restricts to an isomorphism
\[ \iota: H\longrightarrow \Hom_{K[t]}^\tau( \mot, \tate). \qedhere \]
\end{proof}

\begin{rem}\label{rem:tau-difference-solutions}
Assume for the moment that $E$ is abelian, i.e. that $\mot$ is a free finitely generated $K[t]$-module. Let $\{\mu_1,\ldots,\mu_r\}$ be the basis of $\Hom_{K[t]}( \mot,K[t])$
which is dual to the chosen basis $\{m_1\ldots, m_r\}$ of $\mot$.
Then for $x_1,\ldots,x_r\in \tate$, the homomorphism $\mu=\sum_{i=1}^r x_i\mu_i\in \Hom_{K[t]}( \mot,\tate)$ is $\tau$-equivariant if and only if for all $j=1,\ldots, r$,
\[ \tau(\mu(m_j))=\tau(x_j)\]
equals
\[ \mu(\tau(m_j))=\mu\biggl( \Biggl(\Theta\svect{m}{r}\Biggr)_j\biggr)=\Biggl( \Theta \begin{pmatrix}
\mu(m_1) \\ \vdots \\ \mu(m_r) \end{pmatrix}\Biggr)_j =\Biggl( \Theta \svect{x}{r}\Biggr)_j .\]
Hence, the composition $\iota\circ \delta$ gives an $\FF_q[t]$-isomorphism between
the lattice $\Lambda$ and the solutions of the $\tau$-difference equation
\[   \tau\Bigl( \svect{x}{r}\Bigr) =  \Theta \svect{x}{r}.\]

In the case of Drinfeld modules, by choosing the basis $m_1=1,\ldots, m_r=\tau^{r-1}$, these solutions are just the
vectors $\transp{(g, \tau(g),\ldots, \tau^{r-1}(g))}$ where as in Remark \ref{rem:delta-for-drinfeld-modules}, $g=g_{\phi}(\lambda;t)$ is the Anderson generating function associated to $\lambda\in \Lambda$.
\end{rem}

\section{The dual $t$-motive and the forth isomorphism}\label{sec:dual-t-motive}

In this section, $E$ is still an arbitrary $t$-module.\\
We consider the dual $t$-motive $\dumot=\Hom_{\mathrm{grp},\FF_q}(\GG_{a,\CC_\infty},E_{\CC_\infty})$ over $\CC_\infty$ associated to the $t$-module $E$ with $\FF_q[t]$-action via $a\cdot m:=\phi_a\circ m$ for all $m\in \dumot$ and $a\in \FF_q[t]$,
and with $\CC_\infty\{\sigma\}$-action via $\psi\cdot m:=m\circ \psi^*$ for all $m\in \dumot$ and $\psi\in \CC_\infty\{\sigma\}$ where for $\psi=\sum_i a_i\sigma^i\in \CC_\infty\{\sigma\}$, one defines 
$\psi^*=\sum_i \tau^i a_i=\sum_i a_i^{q^{i}}\tau^i\in \CC_\infty\{\tau\}=\End_{\mathrm{grp},\FF_q}(\GG_{a,\CC_\infty})$.

The aim of this section is to establish an isomorphism $H\to (\dumot \otimes_{\CC_\infty[t]} \tate)^\sigma$.

\begin{rem}\label{rem:dual-motive-to-H-in-HJ}
Since, we already know that $H$ is isomorphic to compatible systems of $\phi_{t^i}$-torsion points in $E(\CC_\infty)$, the isomorphism $H\to (\dumot \otimes_{\CC_\infty[t]} \tate)^\sigma$ that we will obtain in 
Theorem \ref{thm:H-isom-dual-m-tate} below, is just a special
case of the canonical bijection given in  \cite[Thm.~5.18]{uh-akj:pthshcff}.

However, in \cite[Thm.~5.18]{uh-akj:pthshcff}, they assume that $\dumot$ is finitely generated as $\CC_\infty[t]$-module which we don't, and our approach gives a natural description of this isomorphism.

In the case of a Drinfeld module $E$, a construction of a basis of $(\dumot \otimes_{\CC_\infty[t]} \tate)^\sigma$
via the basis of $H$ consisting of Anderson generating functions is given in \cite[\S 4.2]{fp:aiacnn}. 
\end{rem}

The starting point for getting the desired isomorphism is the sequence of $\FF_q[t]$-modules
\begin{equation}\label{eq:ses-dual-motive}
   0\to \dumot \xrightarrow{\sigma-\id} \dumot \xrightarrow{\ev} E(\CC_\infty) \to 0.
\end{equation}
where $\ev$ is defined by $\ev(x)=x(1)$ for all $x\in \dumot=\Hom_{\mathrm{grp},\FF_q}(\GG_{a,\CC_\infty},E_{\CC_\infty})$.
This sequence is exact (cf.~e.g.~\cite[Prop.~5.6]{uh-akj:pthshcff}), and we sketch the proof of the exactness here, since we will need to refer to it later:\\
First at all, fix an isomorphism $E\cong \GG_a^d$, and let $\kappa_1,\ldots, \kappa_d:E\to \GG_a$ be the
corresponding coordinate functions. Further, let $(\check{\kappa}_1,\check{\kappa}_2,\ldots,\check{\kappa}_d)$ be the $\CC_\infty\{\sigma\}$-basis of $\dumot$ dual to $(\kappa_1,\ldots, \kappa_d)$, i.e.~
$\kappa_j\circ \check{\kappa}_i=0$ for $i\neq j$ and $\kappa_j\circ \check{\kappa}_j=\id_{\GG_a}$ for all $j=1,\ldots, d$. Be aware that also $\sum_{j=1}^d \check{\kappa}_j\circ \kappa_j=\id_E$.

Given an element $0\neq x=\sum_{j=1}^d (\sum_{l=0}^{n_j} a_{j,l}\sigma^l) \check{\kappa}_j\in \dumot$ with $a_{j,l}\in \CC_\infty$, then
\[  (\sigma-\id)(x)=\sum_{j=1}^d (\sum_{l=0}^{n_j} (a_{j,l})^{1/q}\sigma^{l+1}-a_{j,l}\sigma^l) \check{\kappa}_j 
\neq 0.   \]
Hence, $\sigma-\id$ is injective. Furthermore,
\[ \ev\bigl( (\sigma-\id)(x) \bigr)=\ev\bigl( x\circ \tau -x\bigr)=x(\tau(1))-x(1)=0.\]
Therefore, the composition is zero. On the other hand, if
$y=\sum_{j=1}^d (\sum_{l=0}^{n_j} b_{j,l}\sigma^l) \check{\kappa}_j\in \dumot$ such that $\ev(y)=0$,
then for all $j=1,\ldots,d$: $\sum_{l=0}^{n_j} \tau^l b_{j,l}(1)=0$, i.e. 
\[\sum_{l=0}^{n_j} b_{j,l}^{q^l}=0,\]
and one easily checks that the element
\[ x=\sum_{j=1}^d \left( \sum_{l=0}^{n_j} \bigl(-b_{j,l}-b_{j,l-1}^{q^{-1}}-\ldots - b_{j,1}^{q^{-l+1}}- b_{j,0}^{q^{-l}} \bigr)\sigma^l  \right)  \check{\kappa}_j\in \dumot \]
is a preimage of $y$ under $\sigma-\id$. Hence, the sequence is exact in the middle.\\
For showing that $\ev$ is surjective, we just have to recognize that for any $e\in E(\CC_\infty)$, a preimage under
$\ev$ is given by $x=\sum_{j=1}^d  \kappa_j(e)\cdot \check{\kappa}_j$, since
\[  \ev\left(  \sum_{j=1}^d  \kappa_j(e)\cdot \check{\kappa}_j\right)
= \sum_{j=1}^d \check{\kappa}_j\bigl( \kappa_j(e)\cdot 1\bigr)= \left( \sum_{j=1}^d \check{\kappa}_j\circ \kappa_j\right) (e)=e.\]

\medskip

\begin{prop}
The sequence \eqref{eq:ses-dual-motive} induces a short exact sequence
\begin{equation}\label{eq:ses-dual-motive-tate}
   0\to \dumot\otimes_{\CC_\infty} \tate \xrightarrow{\sigma\otimes \sigma-\id} \dumot\otimes_{\CC_\infty} \tate \xrightarrow{\ev\cs{t}} E\cs{t} \to 0
\end{equation}
where $\ev\cs{t}\left(x\otimes (\sum_i f_it^i)\right):=\sum_i \ev(f_ix) t^i$.
\end{prop}

\begin{proof}
By applying the maps in  the sequence \eqref{eq:ses-dual-motive} coefficient-wise, we obviously obtain a short exact sequence of formal power series
\[ 0\to \dumot\ps{t}  \xrightarrow{(\sigma-\id)\ps{t}} \dumot\ps{t} \xrightarrow{\ev\ps{t}} E\ps{t}\to 0,\]
 and the objects in the sequence \eqref{eq:ses-dual-motive-tate} can be seen
as $\CC_\infty$-subspace, and its maps are just the restrictions of the maps $(\sigma-\id)\ps{t}$ and $\ev\ps{t}$.
Therefore, $\sigma\otimes \sigma-\id$ is injective, and the composition $\ev\cs{t}\circ (\sigma\otimes \sigma-\id)$ is zero. The map  $\ev\cs{t}$ is surjective, since similar to the computation above, we see that for
$\sum_{i\geq 0} e_it^i\in E\cs{t}$, a preimage under $\ev\cs{t}$ is
given by $\sum_{j=1}^d \check{\kappa}_j\otimes \left( \sum_i \kappa_j(e_i)t^i \right)\in \dumot\otimes_{\CC_\infty} \tate$.
For showing exactness in the middle, we recognize that every element $y\in \dumot\otimes_{\CC_\infty} \tate$ can be written in the form $y=\sum_{j=1}^d \sum_{l=0}^n \sigma^l\check{\kappa}_j\otimes \left(\sum_i f_{jli}t^i\right)$
for some $n$ and $\sum_i f_{jli}t^i\in \tate$. If $\ev\cs{t}(y)=0$, then as above, $\sum_{l=0}^{n} f_{jli}^{q^l}=0$
for all $j=1,\ldots, d$ and all $i$, and one obtains a preimage under $\sigma\otimes \sigma-\id$
as $x=\sum_{j=1}^d \sum_{l=0}^n \sigma^l\check{\kappa}_j\otimes \left(\sum_i g_{jli}t^i\right)\in \dumot\otimes_{\CC_\infty} \powser$ where
\[ g_{jli}:=-f_{jli}-f_{j,l-1,i}^{q^{-1}}-\ldots - f_{j0i}^{q^{-l}}. \]
Since for all $j=1,\ldots, d$ and $l=0,\ldots, n$
\[ \lim_{i\to \infty} \betr{g_{jli}}\leq  \lim_{i\to \infty}  \max_{0\leq \nu\leq l} \{ \betr{f_{j\nu i}^{q^{\nu-l}}} \}
=\max_{0\leq \nu\leq l} \{ \lim_{i\to \infty} \betr{f_{j\nu i}}^{q^{\nu-l}} \}=0,\]
the series $\sum_i g_{jli}t^i$ are indeed in the Tate algebra $\tate$.
\end{proof}

Defining the operator $\phi_t-t$ on $\dumot\otimes_{\CC_\infty} \tate$  to be $\phi_t\otimes \id-\id\otimes t$,
 we obtain a commutative diagram of $\FF_q[t]\otimes_{\FF_q} \FF_q[t]$-modules with exact rows
 
 \centerline{
\xymatrix@C+6mm{
 0 \ar[r] & \dumot\otimes_{\CC_\infty} \tate \ar[d]^{\phi_t-t} \ar[r]^{\sigma\otimes \sigma-\id} 
 & \dumot\otimes_{\CC_\infty} \tate  \ar[d]^{\phi_t-t} \ar[r]^(.65){\ev\cs{t}} &  
 E\cs{t}\ar[d]^{\phi_t-t}\ar[r]  & 0 \\
 0 \ar[r] & \dumot\otimes_{\CC_\infty} \tate \ar[r]^{\sigma\otimes \sigma-\id}  
 & \dumot\otimes_{\CC_\infty} \tate \ar[r]^(.65){\ev\cs{t}}   & E\cs{t}\ar[r]  & 0 .
}}

\begin{thm}\label{thm:H-isom-dual-m-tate} \
\begin{enumerate}
\item The homomorphism $\phi_t-t:\dumot\otimes_{\CC_\infty} \tate\to \dumot\otimes_{\CC_\infty} \tate$ is injective.
\item $\Coker(\phi_t-t)=\dumot\otimes_{\CC_\infty[t]} \tate$.
\item Using the snake lemma, the diagram induces an isomorphism of $\FF_q[t]$-modules
\[ H\longrightarrow \bigl( \dumot\otimes_{\CC_\infty[t]} \tate \bigr)^\sigma,\]
where $()^\sigma$ denotes the $\sigma$-invariant elements.
\end{enumerate}
\end{thm}

\begin{proof}
\begin{enumerate}
\item First consider the $t$-action on $\dumot$ via $\phi_t$. If for an element $x\in \dumot=\Hom_{\mathrm{grp},\FF_q}(\GG_{a,\CC_\infty},E_{\CC_\infty})$ we
have $\phi_t\circ x=0$, then the image of $x$ has to be in the kernel of $\phi_t:E\to E$. However, the kernel of $\phi_t$ is finite, and the image of $x$ is connected. Hence, the composition $\phi_t\circ x$ can only be zero, if 
the image of $x$ is trivial, i.e. if $x=0$. Hence, $\phi_t$ is injective on $\dumot$.

Assume that $\sum_{j=1}^l x_j\otimes (\sum_{i\geq 0} f_{j,i}t^i)\in \dumot\otimes_{\CC_\infty} \tate$ is in the kernel of $\phi_t-t$, hence
\[ \sum_{j=1}^l (\phi_t\circ x_j)\otimes (\sum_{i\geq 0} f_{j,i}t^i )
= \sum_{j=1}^l x_j\otimes (\sum_{i\geq 0} f_{j,i}t^{i+1}). \]
By comparing the coefficients of the various $t$-powers, we therefore obtain
\[  \sum_{j=1}^l \phi_t\circ x_j\cdot f_{j,0}=0\quad\text{and}\quad \sum_{j=1}^l \phi_t\circ x_j\cdot f_{j,i}=\sum_{j=1}^l x_j\cdot f_{j,i-1}\quad \text{ for all }i\geq 1.\]
Then inductively, we obtain that $\sum_{j=1}^l  x_j\cdot f_{j,i}=0$ for $i\geq 0$, since 
$\phi_t$ is injective on $\dumot$.
\item The cokernel is just the quotient on which both $t$-actions agree. Hence, 
\[\Coker(\phi_t-t)=\dumot\otimes_{\CC_\infty[t]} \tate. \]
\item Adding kernels and cokernels to the diagram above, we obtain

\centerline{
\xymatrix@C+4mm{
 & 0 \ar[d] &  0  \ar[d] \ar[r] &  H \ar[d] 
    \ar@{-->}'[r] `r[d]`[d]+/d 3ex/ `^d[lll]+/l 3ex/ `[dddll]      [dddll] 
 & \\
 0 \ar[r] & \dumot\otimes_{\CC_\infty} \tate \ar[d]^{\phi_t-t} \ar[r]^{\sigma\otimes \sigma-\id} 
 &  \dumot\otimes_{\CC_\infty} \tate  \ar[d]^{\phi_t-t} \ar[r]^(.65){\ev\cs{t}} &  
 E\cs{t}\ar[d]^{\phi_t-t}\ar[r]  & 0 \\
 0 \ar[r] & \dumot\otimes_{\CC_\infty} \tate \ar[r]^{\sigma\otimes \sigma-\id} \ar[d] & \dumot\otimes_{\CC_\infty} \tate \ar[r]^(.65){\ev\cs{t}} \ar[d]  & E\cs{t}\ar[r]  & 0\\
& \dumot\otimes_{\CC_\infty[t]} \tate \ar[r]^{\sigma-\id} & \dumot\otimes_{\CC_\infty[t]} \tate & &
}}
As the snake sequence is exact, we obtain the desired isomorphism
\[  H\longrightarrow \Ker(\sigma-\id)= \bigl(\dumot\otimes_{\CC_\infty[t]} \tate\bigr)^\sigma.  \]
\end{enumerate}
\end{proof}

\begin{rem}\label{rem:iso-dual-motive-to-H}
By starting with the commuting diagram with exact rows 

\centerline{
\xymatrix@C+3mm{
 0 \ar[r] & \dumot\otimes_{\CC_\infty} \tate \ar[r]^{\phi_t-t} \ar[d]^{\sigma\otimes \sigma-\id} 
 &  \dumot\otimes_{\CC_\infty} \tate  \ar[r]  \ar[d]^{\sigma\otimes \sigma-\id} &  \dumot\otimes_{\CC_\infty[t]} \tate \ar[d]^{\sigma-\id}\ar[r] 
 & 0 \\
 0 \ar[r] & \dumot\otimes_{\CC_\infty} \tate \ar[r]^{\phi_t-t}  & \dumot\otimes_{\CC_\infty} \tate \ar[r]   & \dumot\otimes_{\CC_\infty[t]} \tate \ar[r] & 0 
}}
and applying the snake lemma, one obtains the inverse isomorphism
\[   \bigl(\dumot\otimes_{\CC_\infty[t]} \tate \bigr)^\sigma\longrightarrow 
\Ker\bigl(\phi_t-t:E\cs{t}\to E\cs{t}\bigr)= H \]
to the isomorphism above.
\end{rem}

\begin{rem}
Working with $\dumot\otimes_{\CC_\infty}\powser$ etc. instead of $\dumot\otimes_{\CC_\infty} \tate$, one obtains in the same way an isomorphism
\[  \hat{H}\longrightarrow \left(\dumot\otimes_{\CC_\infty[t]} \powser\right)^\sigma. \]
\end{rem}

\section{The matrix which specializes to the periods}\label{sec:matrix-for-specialization}

The task of this section is to show how one obtains the coordinates of a basis of the period lattice as special values
of some matrix which meets the conditions of the ABP-criterion resp.~of Theorem \ref{thm:conseq-of-abp}.

\medskip

Throughout the whole section, $E$ is a uniformizable abelian $t$-module over $K$ of dimension $d$ and $\mot$ the associated $t$-motive.
We fix a $K[t]$-basis $\vect{m}=\transp{(m_1,\ldots, m_r)}$
 of the $t$-motive $\mot$, we let $\Theta\in \Mat_{r\times r}(K[t])$ be such that $\tau (\vect{m})=\Theta \vect{m}$, and fix $\Upsilon\in \GL_{r}(\tate)$ a rigid analytic trivialization of $\mot$ with respect to $\vect{m}$, i.e.~$\tau(\Upsilon \vect{m})=\Upsilon \vect{m}$ or in other words,
$\tau(\Upsilon)=\Upsilon\Theta^{-1}$.

\begin{prop}\label{prop:def-of-r}
For any $C\in \GL_r(K[t])$, let $\tilde{\Theta}=C\Theta\sigma(C)^{-1}\in \Mat_{r\times r}(\bar{K}[t])$ and $R:=\tau(\Upsilon)C^{-1}\in \GL_r(\tate)$.
Then for any $l\geq 0$, the pair $(\rho_{[l]}(\tilde{\Theta}),\rho_{[l]}(R))$ meets the conditions of $\Phi$ and $\Psi$ in Thm.~\ref{thm:conseq-of-abp}, i.e.
\begin{align*}
\det(\rho_{[l]}(\tilde{\Theta})) & =c_l\cdot (t-\theta)^{s_l} \text{ for some }c_l\in \bar{K}^\times, s_l\geq 1,\\
\sigma\bigl(\rho_{[l]}(R)\bigr) & = \rho_{[l]}(R)\cdot \rho_{[l]}(\tilde{\Theta}).
\end{align*}
\end{prop}

\begin{proof}
As $C\in \GL_r(K[t])$, its determinant $\det(C)$ is a unit in $K[t]$, hence $\det(C)\in K^\times$. Since by definition of a $t$-motive $\det(\Theta)=c(t-\theta)^s$ for some $c\in K^\times$ and $s\geq 1$, we have
\[ \det(\tilde{\Theta})=\det(C)\det(\Theta)\sigma(\det(C))^{-1}
=c_0(t-\theta)^s\]
where $c_0=c\cdot \det(C)\cdot \sigma(\det(C))^{-1}\in \bar{K}^\times$.
Furthermore by definition,
\[  \sigma(R)=\sigma( \tau(\Upsilon)C^{-1})=\Upsilon \sigma(C)^{-1}=\tau(\Upsilon)\Theta\sigma(C)^{-1}=R\cdot \tilde{\Theta}.\]
Let now $l\geq 0$ be arbitrary. By definition of $\rho_{[l]}$, one has $\det(\rho_{[l]}(\tilde{\Theta}))=\det(\tilde{\Theta})^{l+1}$ which implies the first equation.
The second equation is just a consequence of the fact that $\rho_{[l]}$ is multiplicative and commutes with~$\sigma$.
\end{proof}

\begin{thm}\label{thm:periods-as-special-values}
There is a matrix $B\in \GL_r(K[t])$, and a coordinate system of $E$ with corresponding isomorphism $E\cong \GG_{a}^d$
such that the coordinates of a basis of the period lattice are the values at $t=\theta$ 
of certain entries of the matrix $\rho_{[d-1]}( R^{-1})=\rho_{[d-1]}( R)^{-1}$ where $R:=\tau(\Upsilon)B^{-1}\in \GL_r(\tate)$.
\end{thm}

The proof takes the rest of the section. Actually, we will see that the proof is constructive and
shows in explicit examples which matrix entries are relevant.

%
%
%

\begin{prop}\label{prop:base-change-motive-to-coordinates}
Let $\kappa_1,\ldots, \kappa_d:E\to \GG_a$ be the coordinate functions corresponding to a choice of coordinates $E(K)\cong K^d$, and $A\in \Mat_{d\times r}(K[t])$ be such that
\[  \svect{\kappa}{d} = A\cdot \svect{m}{r}\]
(which exists, since the coordinate functions are special elements of $\mot$). 
Let $\Upsilon\in \GL_{r}(\tate)$ be the rigid analytic trivialization of $\mot$ fixed above.

Then an $\FF_q[t]$-basis of $H\subseteq E\cs{t}\cong (\tate)^d$ is given by the columns of
$ A\cdot \Upsilon^{-1}. $
\end{prop}

\begin{proof}
If $h=\sum_{i}e_it^i\in H$ with corresponding element $\mu\in \Hom_{K[t]}^\tau( \mot, \tate)$ via the isomorphism $\iota$ in Theorem~\ref{thm:h-isom-M-tate-dual}, then in coordinates we have
\[  h=\left(\begin{smallmatrix}\sum_{i}\kappa_1(e_i)t^i \\
\vdots \\ \sum_{i}\kappa_d(e_i)t^i 
\end{smallmatrix}\right)=\left(\begin{smallmatrix} \mu(\kappa_1)
\\ \vdots \\ \mu(\kappa_d)\end{smallmatrix}\right)
= A\cdot \left(\begin{smallmatrix} \mu(m_1)
\\ \vdots \\ \mu(m_r)\end{smallmatrix}\right).\]
Let $\{\mu_1,\ldots,\mu_r\}$ be the basis of $\Hom_{K[t]}( \mot,K[t])\subseteq \Hom_{K[t]}( \mot, \tate)$ which is dual to $\vect{m}$, then the vector 
$ \left(\begin{smallmatrix} \mu(m_1)
\\ \vdots \\ \mu(m_r)\end{smallmatrix}\right)$ is just the representation of $\mu$ in the basis $\{\mu_1,\ldots,\mu_r\}$. 
Furthermore, since $\tau(\Upsilon)=\Upsilon\Theta^{-1}$, we have
$\tau(\Upsilon^{-1})=\Theta\Upsilon^{-1}$. Therefore, the columns of $\Upsilon^{-1}$ 
are solutions of the $\tau$-difference equation
\[   \tau\Bigl( \svect{x}{r}\Bigr) =  \Theta \svect{x}{r},\]
and hence provide an $\FF_q[t]$-basis of 
$\Hom_{K[t]}^\tau( \mot, \tate)$ with respect to the basis $\{\mu_1,\ldots,\mu_r\}$
(see Rem.~\ref{rem:tau-difference-solutions}).

Hence, an $\FF_q[t]$-basis of $H\subseteq (\tate)^d$ is given by the columns of
$A\cdot \Upsilon^{-1}$.
\end{proof}

\begin{prop}\label{prop:appropriate-choice-of-coordinates}
Let $s_i=\begin{pmatrix} 0&\ldots & 0 & 1 & 0 & \ldots & 0 \end{pmatrix}\in K^{1\times r}$ be the $i$-th standard basis vector for
$1\leq i\leq r$, and let $\Theta\in \Mat_{r\times r}(K[t])$ be such that $\tau (\vect{m})=\Theta \vect{m}$.
Then there
is $B\in \GL_r(K[t])$, a choice of coordinate functions $\kappa_1,\ldots, \kappa_d$ and $A'\in \Mat_{d\times r}(K(t))$
where each row of $A'$ is of the form
\[     (t-\theta)^{-\gamma}\cdot s_i\]
for positive integers $\gamma$ such that 
\[ \svect{\kappa}{d} = A'B\Theta \cdot \svect{m}{r}.\]
\end{prop}

\begin{proof}
Let $\mot_1=K[t]\tau(\mot)$ denote the $K[t]$-submodule of  $\mot$ generated by the image of $\tau$, i.e.~generated by $\tau(m_1),\ldots, \tau(m_r)$.
By definition of a $t$-motive, the quotient $K[t]$-module $\mot/\mot_1$ is a $(t-\theta)$-torsion module.
Therefore by the elementary divisor theorem, there is a $K[t]$-basis $\{ n_1,\ldots, n_r\}$ of $\mot$ and
integers $\alpha_1,\ldots, \alpha_r$ 
such that
$\{(t-\theta)^{\alpha_1}n_1,\ldots, (t-\theta)^{\alpha_r}n_r\}$ is a $K[t]$-basis of $\mot_1$. In particular, the residue classes of 
$\{ (t-\theta)^{\beta_i}n_i \mid 1\leq i\leq r, 0\leq \beta_i<\alpha_i\}$ form a $K$-basis of $\mot/\mot_1$,
and hence $\{ (t-\theta)^{\beta_i}n_i \mid 1\leq i\leq r, 0\leq \beta_i<\alpha_i\}$ is a $K\{\tau\}$-basis of $\mot$. Thus, we can choose the coordinate functions $\kappa_1,\ldots, \kappa_d$ to be these functions, i.e.~each $\kappa_j$ equals some $(t-\theta)^{-\gamma_i}\cdot \left( (t-\theta)^{\alpha_i}n_i\right)$ for appropriate $i$ and $0<\gamma_i\leq \alpha_i$. This means that
\[
 \svect{\kappa}{d} = A' \cdot  \begin{pmatrix} (t-\theta)^{\alpha_1}n_1\\ \vdots \\  (t-\theta)^{\alpha_r}n_r\end{pmatrix} \]
for a matrix $A'$ as given in the statement of the proposition.

On the other hand, $\tau(\vect{m})=\Theta\vect{m}$ is another $K[t]$-basis of $\mot_1$. Hence, there is a base change matrix $B\in \GL_r(K[t])$ such that
\[   \begin{pmatrix} (t-\theta)^{\alpha_1}n_1\\ \vdots\\  (t-\theta)^{\alpha_r}n_r\end{pmatrix}
= B \cdot \Theta \svect{m}{r}.\]
Putting this into the previous equation, leads to
\[ \svect{\kappa}{d} = A'B\Theta \cdot \svect{m}{r}. \qedhere \]
\end{proof}

Combining the previous results with the description of the isomorphism $-\delta$ in Prop.~\ref{prop:inverse-of-delta-explicit}, we obtain the following corollary.

\begin{cor}\label{cor:periods-as-residues}
Let the coordinate functions $\kappa_1,\ldots, \kappa_d$ be chosen as in the previous proposition, and $A'$ and $B$ the corresponding matrices. Then a basis of the period lattice is given (with respect to the chosen coordinates) by the columns of the matrix
\[    \res\left( A'B\Theta\Upsilon^{-1}\, dt\right).
\]
\end{cor}

For relating these residues with values of hyperderivatives, we need the following lemma.

\begin{lem}
Let $f\in \CC_\infty\ls{t-\theta}$, and $l\in \NN$ such that $(t-\theta)^l\cdot f\in \CC_\infty\ps{t-\theta}$. Then
 one has
\[   \res(f\, dt)= \hd{l-1}{(t-\theta)^l\cdot f}\Big|_{t=\theta}. \]
\end{lem}

\begin{proof}
Write  $f=\sum_{j=-l}^\infty c_j (t-\theta)^j\in \CC_\infty\ls{t-\theta}$, then $ \res(f\, dt)=c_{-1}$.
On the other hand,
\begin{eqnarray*}
\hd{l-1}{(t-\theta)^l\cdot f} &=& \hd{l-1}{\sum_{j=-l}^\infty c_j (t-\theta)^{j+l}} 
=  \sum_{j=-1}^\infty c_j \binom{j+l}{l-1} (t-\theta)^{j+1}.
\end{eqnarray*}
Hence,
\[  \hd{l-1}{(t-\theta)^l\cdot f}\Big|_{t=\theta} = \left(  \sum_{j=-1}^\infty c_j \binom{j+l}{l-1} (t-\theta)^{j+1}\right)\Big|_{t=\theta} = c_{-1}. \]
\end{proof}

\begin{proof}[Proof of Theorem \ref{thm:periods-as-special-values}]
By Corollary \ref{cor:periods-as-residues}, we already know that a basis of the period lattice is given by the columns
of
\[    \res\left( A'B\Theta\Upsilon^{-1}\, dt\right) \]
with an appropriate choice of the coordinate
functions $\kappa_1,\ldots, \kappa_d$ and corresponding $B\in \GL_r(K[t])$ and $A'\in \Mat_{d\times r}(K(t))$ whose
$j$-th row is
\[     (t-\theta)^{-\gamma_j}\cdot s_{i_j}\]
with $s_{i_j}$ being the $i=i_j$-th standard basis vector $s_i=\begin{pmatrix} 0&\ldots & 0 & 1 & \ldots & 0 \end{pmatrix}$ and $0<\gamma_j\leq d$.
Hence the $j$-th row of $A'B\Theta\Upsilon^{-1}$ is the $(t-\theta)^{-\gamma_j}$-multiple of the $i_j$-th row of
$B\Theta\Upsilon^{-1}=B\tau(\Upsilon)^{-1}=R^{-1}$.
Further, $R^{-1}$ does not have a pole at $t=\theta$ (see Remark after Thm.~\ref{thm:conseq-of-abp} and Prop.~\ref{prop:def-of-r}), hence for the $j$-th row of the matrix
$\res\left( A'B\Theta\Upsilon^{-1}\, dt\right)$ we obtain
\begin{eqnarray*}
\res\left( A'B\Theta\Upsilon^{-1}\, dt\right)_j &=&  \res\left( (t-\theta)^{-\gamma_j} R^{-1}\, dt\right)_{i_j}\\
&=&   \hd{\gamma_j-1}{(t-\theta)^{\gamma_j}\cdot  (t-\theta)^{-\gamma_j} R^{-1}}_{i_j}\big|_{t=\theta} \\
&=&  \hd{\gamma_j-1}{R^{-1}}_{i_j}\big|_{t=\theta}.
\end{eqnarray*}

Therefore, each row of $\res\left( A'B\Theta\Upsilon^{-1}\, dt\right)$ is a row in the matrix
\[  \begin{pmatrix}  \hd{d-1}{R^{-1}} \\ \vdots \\   \hd{1}{R^{-1}} \\ R^{-1}
\end{pmatrix}\Big|_{t=\theta} \, . \]
The latter, however, is just the $\bigl( r(d+1)\times r\bigr)$-submatrix of $\rho_{[d-1]}(R^{-1})|_{t=\theta}$ consisting of the last $r$ columns.
\end{proof}

\begin{rem}\label{rem:simpler-description-of-r}
From the presentation in the form $\rho_{[d-1]}(R^{-1})|_{t=\theta}$ we are enabled to use the ABP-criterion for showing transcendence results. In explicit examples, the matrix $R^{-1}=B\Theta\Upsilon^{-1}$ is often described easier. Namely, if
$\{m_1,\ldots, m_r\}$ already is a basis of $\mot$ such that $\{(t-\theta)^{\alpha_1}m_1,\ldots, (t-\theta)^{\alpha_r}m_r\}$ is a $K[t]$-basis of $K[t]\tau(\mot)$, then $B$ can be chosen such that
\[   \begin{pmatrix} (t-\theta)^{\alpha_1}m_1\\ \vdots\\  (t-\theta)^{\alpha_r}m_r\end{pmatrix}
= B \cdot \Theta \svect{m}{r}.\]
Hence, $B\cdot \Theta=  \textrm{diag}\left((t-\theta)^{\alpha_1} , (t-\theta)^{\alpha_2} , \ldots,  (t-\theta)^{\alpha_r}\right)$
and 
\[ R^{-1}=B\Theta \Upsilon^{-1}= \textrm{diag}\left((t-\theta)^{\alpha_1} , (t-\theta)^{\alpha_2} , \ldots,  (t-\theta)^{\alpha_r}\right)\cdot  \Upsilon^{-1}.\]
\end{rem}

\section{Examples}

We illustrate two examples:
\begin{exmp}
If $E=D$ is a Drinfeld module over $K$ of rank $r$,
the associated $t$-motive is
$\mot(D)\cong K\{\tau\}$ with $K[t]$-basis $m_1=1,m_2=\tau,\ldots, m_r=\tau^{r-1}$,
and the inverse of the rigid analytic trivialization $\Upsilon$ is
\[ \Upsilon^{-1}=\begin{pmatrix} g_1 & g_2 & \cdots & g_r \\
\tau(g_1) & \tau(g_2) & \cdots & \tau(g_r)\\
\vdots & & & \vdots \\
\tau^{r-1}(g_1) & \tau^{r-1}(g_2) & \cdots & \tau^{r-1}(g_r)
\end{pmatrix}, \]
for the Anderson generating functions $g_1,\ldots, g_r\in \tate$ associated to a
basis $\lambda_1,\ldots, \lambda_r$ of the lattice $\Lambda$.
 As $K[t]\tau(\mot(D))$ is also generated by 
 $\{(t-\theta)m_1,m_2,\ldots, m_r\}$,  the matrix $R^{-1}$ is just
$R^{-1}=\textrm{diag}\left((t-\theta) , 1 , \ldots, 1\right)\cdot  \Upsilon^{-1},$
by Remark \ref{rem:simpler-description-of-r}.
Furthermore, the coordinate function is $\kappa=(1,0,\ldots,0)\vect{m}$, and hence
we recover the basis of the lattice as the entries in the first row of the matrix

\[ -\rho_{[0]}(R^{-1})|_{t=\theta}=- 
\begin{pmatrix} (t-\theta)g_1 & (t-\theta)g_2 & \cdots & (t-\theta)g_r \\
\tau(g_1) & \tau(g_2) & \cdots & \tau(g_r)\\
\vdots & & & \vdots \\
\tau^{r-1}(g_1) & \tau^{r-1}(g_2) & \cdots & \tau^{r-1}(g_r)
\end{pmatrix}\Big|_{t=\theta} , \]
which perfectly fits with the well-known facts.
\end{exmp}

\begin{exmp}\label{ex:carlitz-tensor-power}
Let $E=C^{\otimes n}$ be the $n$-th tensor power of the Carlitz module which is a $t$-module of dimension $n$ and rank $1$. Its $t$-motive is $\mot(C^{\otimes n})=K[t]\cdot m$ for a basis element $m$ with $\tau$-action given by
\[  \tau(m)=(t-\theta)^n m.\]
A rigid analytic trivialization is given by $\Upsilon=\omega(t)^{-n}$ where
$\omega$ is the Anderson-Thakur function with $-\res(\omega\, dt)=\pitilde$ being the Carlitz period. 
The coordinate functions can be chosen to be $\kappa_i=(t-\theta)^{i-1}m$ for $i=1,\ldots, n$, and hence a basis for the period lattice is given by the vector
\[  \svect{z}{n}=\begin{pmatrix}
 \res (\omega(t)^n) \\
 \res ((t-\theta)\omega(t)^n) \\
\vdots \\
 \res ((t-\theta)^{n-1}\omega(t)^n) 
\end{pmatrix} \]
for $i=1,\ldots, n$. Up to this point, this is already given in \cite[\S 2.5]{ga-dt:tpcmzv}.
However, recognizing that this vector is the last column of 
\[  \rho_{[n-1]}(R^{-1})|_{t=\theta}=  \rho_{[n-1]}((t-\theta)^{n}\omega(t)^n)|_{t=\theta}, \]
and that $\rho_{[n-1]}(R)$ is a rigid analytic trivialization of a dual $t$-motive, 
enabled us in \cite{am:ptmaip} to show that $z_1,\ldots, z_n$ are algebraically independent over $K$ if $n$ is prime to the characteristic of $K$.
\end{exmp}



\bibliographystyle{plain}

\def\cprime{$'$}

\vspace*{.5cm}

\parindent0cm

\end{document}